\documentclass[a4paper,final]{article}

\usepackage{amsthm}
\newtheorem{teo}{Theorem}

\newtheorem{lem}[teo]{Lemma}
\newtheorem{cor}[teo]{Corollary}

\theoremstyle{definition}
\newtheorem{defi}[teo]{Definition}
\newtheorem{rem}[teo]{Remark}

\usepackage{amssymb}

\usepackage[dvips]{graphicx}

\newcommand{\RP}{\mathbb{RP}}
\newcommand{\matN}{\mathbb{N}}
\newcommand{\matR}{\mathbb{R}}
\newcommand{\matS}{\mathbb{S}}
\newcommand{\matH}{\mathbb{H}}
\newcommand{\matE}{\mathbb{E}}

\newcommand{\calC}{{\cal C}}
\newcommand{\calR}{{\cal R}}
\newcommand{\calT}{{\cal T}}

\newcommand{\ptwoirred}{$\mathbb{P}^2$-irreducible}

\newcommand{\Lthreeone}{L_{3,1}}
\newcommand{\Lfourone}{L_{4,1}}
\newcommand{\twoRPtwo}{2\times\RP^2}

\newcommand{\timtil}{\begin{picture}(12,12)
\put(2,0){$\times$}\put(2,4.5){$\sim$}\end{picture}}

\newcommand{\chiorb}{\chi^{\rm orb}}

\newcommand{\co}{\colon\thinspace}

\title{A 3-manifold complexity via immersed surfaces\footnotetext{This research has been supported by the grant ``Ennio De Giorgi'' (2007-2008) from the Department of Mathematics of the University of Salento.}}

\author{Gennaro {\sc Amendola}}

\begin{document}

\maketitle

\begin{abstract}
We define an invariant, which we call surface-complexity, of closed 3-manifolds by means of Dehn surfaces.
The surface-complexity of a manifold is a natural number measuring how much the manifold is complicated.
We prove that it fulfils interesting properties: it is subadditive under connected sum and finite-to-one on \ptwoirred\ manifolds.
Moreover, for \ptwoirred\ manifolds, it equals the minimal number of cubes in a cubulation of the manifold, except for the sphere $S^3$, the projective space $\RP^3$ and the lens space $\Lfourone$, which have surface-complexity zero.
We will also give estimations of the surface-complexity by means of triangulations, Heegaard splittings, surgery presentations and Matveev complexity.
\end{abstract}

\vspace{1pt}

\begin{center}
\begin{small}
{\bf Keywords}\\[4pt]
3-manifold, complexity, immersed surface, cubulation.

\vspace{.5cm}

{\bf MSC (2000)}\\[4pt]
57M27 (primary), 57M20 (secondary).
\end{small}
\end{center}

\section*{Introduction}

The problem of filtering closed 3-manifolds in order to study them systematically has been approached by many mathematicians.
The aim is to find a function from the class of closed 3-manifolds to the set of natural numbers.
The number associated to a closed 3-manifold should be a measure of how much the manifold is complicated.
For closed surfaces, this can be achieved by means of genus.
For closed 3-manifolds, the problem has been studied very much and many possible functions has been found.
For example, the Heegaard genus, the Gromov norm, the Matveev complexity have been considered.

All these functions fulfil many properties.
For instance, they are additive under connected sum.
However, some of them have drawbacks.
The Heegaard genus and the Gromov norm are not finite-to-one, while the Matveev complexity is.
Hence, in order to carry out a classification process, the latter one is more suitable than the former ones.
The Matveev complexity is also a natural measure of how much the manifold is complicated, because if a closed 3-manifold is \ptwoirred\ and different from the sphere $S^3$, the projective space $\RP^3$ and the Lens space $\Lthreeone$, then its Matveev complexity is the minimum number of tetrahedra in a triangulation of the manifold (the Matveev complexity of $S^3$, $\RP^3$ and $\Lthreeone$ is zero).
Such functions could also be tools to give proofs by induction.
For instance, the Heegaard genus was used by Rourke to prove by induction that every closed orientable 3-manifold is the boundary of a compact orientable 4-manifold~\cite{Rourke}.

The aim of this paper is to define another function (we will call {\em surface-complexity}) from the class of closed 3-manifolds to the set of natural numbers, to prove that it fulfils some properties, to give bounds for it, and to start an enumeration process (we will give a complete list of closed 3-manifolds with complexity one in a subsequent paper~\cite{Amendola:next}).

In~\cite{Vigara:calculus} Vigara used triple points of particular transverse immersions of connected closed surfaces to define the triple point spectrum of a 3-manifold.
The definition of the surface-complexity is similar to Vigara's one, but it has the advantage of being more flexible.
This flexibility will allow us to prove many properties fulfilled by the surface-complexity.

We now sketch out the definition and the results of this paper.
The surface-complexity of a closed 3-manifold will be defined by means of {\em quasi-filling Dehn surfaces} ({\em i.e.}~images of transverse immersions of closed surfaces that divide the manifold into balls).
\begin{description}
\item[Definition.] The surface-complexity $sc(M)$ of a closed 3-manifold $M$ is the minimal number of triple points of a quasi-filling Dehn surface of $M$.
\end{description}

Three properties we will prove are the following ones.
\begin{description}
\item[Finiteness.]
For any integer $c$ there exists only a finite number of connected closed \ptwoirred\ 3-manifolds having surface-complexity $c$.
\item[Naturalness.]
The surface-complexity of a connected closed \ptwoirred\ 3-manifold $M$, different from $S^3$, $\RP^3$ and $\Lfourone$, is the minimal number of cubes in a cubulation of $M$.
The surface complexity of $S^3$, $\RP^3$ and $\Lfourone$ is zero.
\item[Subadditivity.]
The complexity of the connected sum of closed 3-manifolds is less than or equal to the sum of their complexities.
\end{description}
The naturalness property will follow from the features of {\em minimal} quasi-filling Dehn surfaces of connected closed \ptwoirred\ 3-manifolds, where minimal means ``with a minimal number of triple points''.
We will call a quasi-filling Dehn surface of a 3-manifold $M$ {\em filling}, if its singularities induce a cell-decomposition of $M$.
The cell-decomposition dual to a filling Dehn-surface of $M$ is actually a cubulation of $M$.
Hence, in order to prove the naturalness property, we will prove that every connected closed \ptwoirred\ 3-manifold, different from $S^3$, $\RP^3$ and $\Lfourone$, has a minimal filling Dehn surface.
We point out that not all the minimal quasi-filling Dehn surfaces of \ptwoirred\ 3-manifolds are indeed filling.
However, they can be all constructed starting from filling ones (except for $S^3$, $\RP^3$ and $\Lfourone$, for which non-filling ones must be used) and applying a simple move, we will call {\em bubble-move}.

The surface-complexity is related to the Matveev complexity.
Indeed, if $M$ is a connected closed \ptwoirred\ 3-manifold different from $\Lthreeone$ and $\Lfourone$, the double inequality
$\frac{1}{8}c(M) \leqslant sc(M) \leqslant 4c(M)$
holds, where $c(M)$ denotes the Matveev complexity of $M$.
For the sake of completeness, we recall that we have $c(\Lthreeone)=0$, $sc(\Lthreeone)>0$, $c(\Lfourone)>0$ and $sc(\Lfourone)=0$.

The two inequalities above give also estimates of the surface-complexity.
In general, an exact calculation of the surface-complexity of a closed 3-manifold is very difficult, however it is relatively easy to estimate it.
More precisely, it is quite easy to give upper bounds for it, because constructing a quasi-filling Dehn surface of the manifold with the appropriate number of triple points suffices.
With this technique, we will give upper bounds for the surface-complexity of a closed 3-manifold starting from a triangulation, a Heegaard splitting and a surgery presentation on a framed link (in $S^3$) of it.

In the Appendix, we will state some results on closed 3-manifolds with surface-complexity one and we will give some examples.
However, we will postpone the theoretical proof of these results and the classification of the closed 3-manifolds with surface-complexity one to a subsequent paper~\cite{Amendola:next}.
For the sake of completeness, in the Appendix we will also give a brief description of what happens in the 2-dimensional case.
We plan to cope with the 4-dimensional case in a subsequent paper.

\section{Definitions}

Throughout this paper, all 3-manifolds are assumed to be connected and closed.
By $M$, we will always denote such a (connected and closed) 3-manifold.
Using the {\em Hauptvermutung}, we will freely intermingle the differentiable,
piecewise linear and topological viewpoints.

\paragraph{Dehn surfaces}
A subset $\Sigma$ of $M$ is said to be a {\em Dehn surface of $M$}~\cite{Papa}
if there exists an abstract (possibly non-connected) closed surface $S$ and a transverse immersion $f\co S\to M$ such that $\Sigma = f(S)$.

Let us fix for a while $f\co S\to M$ a transverse immersion (hence,
$\Sigma = f(S)$ is a Dehn surface of $M$). By transversality, the number
of pre-images of a point of $\Sigma$ is 1, 2 or 3; so there are three types of
points in $\Sigma$, depending on this number; they are called
{\em simple}, {\em double} or {\em triple}, respectively.
Note that the definition of the type of a point does not depend on the particular
transverse immersion $f\co S\to M$ we have chosen. In fact, the
type of a point can be also defined by looking at a regular neighbourhood (in
$M$) of the point, as shown in Fig.~\ref{fig:neigh_Dehn_surf}.
The set of triple points is denoted by $T(\Sigma)$; non-simple points are called {\em singular} and their set is denoted by $S(\Sigma)$.
\begin{figure}[t]
  \centerline{
  \begin{tabular}{ccc}
    \begin{minipage}[c]{3.5cm}{\small{\begin{center}
        \includegraphics{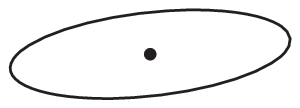}
      \end{center}}}\end{minipage} &
    \begin{minipage}[c]{3.5cm}{\small{\begin{center}
        \includegraphics{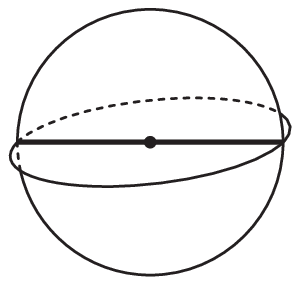}
      \end{center}}}\end{minipage} &
    \begin{minipage}[c]{3.5cm}{\small{\begin{center}
        \includegraphics{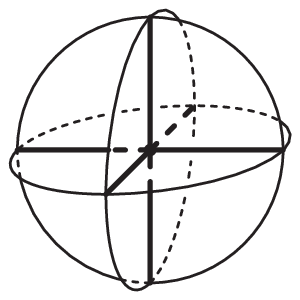}
      \end{center}}}\end{minipage} \\
    \begin{minipage}[t]{3.5cm}{\small{\begin{center}
        Simple\\point
      \end{center}}}\end{minipage} &
    \begin{minipage}[t]{3.5cm}{\small{\begin{center}
        Double\\point
      \end{center}}}\end{minipage} &
    \begin{minipage}[t]{3.5cm}{\small{\begin{center}
        Triple\\point
      \end{center}}}\end{minipage}
  \end{tabular}}
  \caption{Neighbourhoods of points (marked by thick dots) of a Dehn surface.}
  \label{fig:neigh_Dehn_surf}
\end{figure}
From now on, in all figures, triple points are always marked by thick dots and the singular set is also drawn thick.

\begin{rem}
  The topological type of the abstract surface $S$ is determined unambiguously by
  $\Sigma$.
\end{rem}

\paragraph{(Quasi-)filling Dehn surfaces}
A Dehn surface $\Sigma$ of $M$ will be called {\em quasi-filling} if $M \setminus \Sigma$ is made up of balls.
Moreover, $\Sigma$ is called {\em filling}~\cite{Montesinos} if its singularities induce a cell-decomposition of $M$; more precisely,
\begin{itemize}
\item $T(\Sigma) \neq \emptyset$,
\item $S(\Sigma) \setminus T(\Sigma)$ is made up of intervals (called {\em
    edges}),
\item $\Sigma \setminus S(\Sigma)$ is made up of discs (called {\em regions}),
\item $M \setminus \Sigma$ is made up of balls ({\em i.e.}~$\Sigma$ is quasi-filling).
\end{itemize}

Since $M$ is connected and $M\setminus\Sigma$ is made up of (disjoint) balls, the quasi-filling Dehn surface $\Sigma$ is connected.
Consider a small regular neighbourhood $\calR(\Sigma)$ of $\Sigma$ in $M$, then $M\setminus\calR(\Sigma)$ is made up of balls whose closures are disjoint and $M$ can be obtained from $\calR(\Sigma)$ by filling up its boundary components with balls.
Moreover, we have that $M$ minus some balls, {\em i.e.}~$\calR(\Sigma)$, collapses to $\Sigma$.

\begin{rem}\label{rem:Sigma_surf}
Suppose $\Sigma$ is a surface ({\em i.e.}~$S(\Sigma)=\emptyset$).
Then the boundary of $\calR(\Sigma)$ of $\Sigma$ in $M$ is a two-fold covering of $\Sigma$.
Since the boundary of $\calR(\Sigma)$ is made up of spheres and $\Sigma$ is connected, $\Sigma$ is the sphere $S^2$ or the projective plane $\RP^2$.
Therefore, $M$ is the sphere $S^3$ or the projective space $\RP^3$, respectively.
\end{rem}

Let us give some other examples.
Two projective planes intersecting along a loop non-trivial in both of them form a quasi-filling Dehn surface of $\RP^3$, which will be called {\em double projective plane} and denoted by $\twoRPtwo$.
The sphere intersecting a torus (resp.~a Klein bottle) along a loop is a quasi-filling Dehn surface of $S^2\times S^1$ (resp.~$S^2 \timtil S^1$) without triple points.
The {\em quadruple hat} ({\em i.e.}~a disc whose boundary is glued four times along a circle) is a quasi-filling Dehn-surface of the lens-space $\Lfourone$ without triple points.
If we identify the sphere $S^3$ with $\matR^3\cup\{\infty\}$, the three coordinate planes in $\matR^3$, with $\infty$ added, form a filling Dehn surface of $S^3$ with two triple points: $(0,0,0)$ and $\infty$.

It is by now well-known that a filling Dehn surface determines $M$ up to homeomorphism and that every $M$ has standard filling Dehn surfaces (see, for instance, Montesinos-Amilibia~\cite{Montesinos} and Vigara~\cite{Vigara:present}, see also~\cite{Amendola:surf_inv}).
It is not clear how any two standard filling Dehn spheres of the same $M$ are related to each other.
There are only partial results; for instance, we provided in~\cite{Amendola:surf_inv} a finite calculus for nullhomotopic filling Dehn spheres, deducing it
from another one, described by Vigara~\cite{Vigara:calculus}, which has been derived from the more general Homma-Nagase calculus~\cite{Homma-NagaseI, Homma-NagaseII} (see also Hass and Hughes~\cite{Hass-Hughes} and Roseman~\cite{Roseman}).

\paragraph{Abstract filling Dehn surfaces}
A filling Dehn surface $\Sigma$ of $M$ is contained $M$.
However, we can think of if as an abstract cell complex.
For the sake of completeness, we point out that the abstract cell complex $\Sigma$ determines $M$ (and the abstract surface $S$ such that $\Sigma=f(S)$ where $f \co S \rightarrow M$) up to homeomorphism.
The proof of this fact is quite easy (and not strictly connected with the aim of this paper), so we leave it to the reader.

\paragraph{Surface-complexity}
The surface-complexity of $M$ can now be defined as the minimal number of triple points of a quasi-filling Dehn surface of $M$. More precisely, we give the following.
\begin{defi}
The {\em surface-complexity} $sc(M)$ of $M$ is equal to $c$ if $M$ possesses a quasi-filling Dehn surface with $k$ triple points and has no quasi-filling Dehn surface with less than $k$ triple points.
In other words, $sc(M)$ is the minimum of $|T(\Sigma)|$ over all quasi-filling Dehn surfaces $\Sigma$ of $M$.
\end{defi}

We will classify the 3-manifolds having surface-complexity zero in the following section.
At the moment, we can only say that $S^3$, $\RP^3$, $S^2\times S^1$, $S^2 \timtil S^1$ and $\Lfourone$ have surface-complexity zero, because we have seen above that they have quasi-filling Dehn surfaces without triple points.

\paragraph{Triple point spectrum}
For the sake of completeness, we give also Vigara's definition of the triple point spectrum~\cite{Vigara:calculus}.
The {\em triple point spectrum} of $M$ is a sequence of integers $t_i(M)$, with $i\in\matN$, such that $t_i(M)$ is the minimal number of triple points of a filling Dehn surface with genus $i$ of $M$.

\section{Minimality and finiteness}

A quasi-filling Dehn surface $\Sigma$ of $M$ is called {\em minimal} if it has a minimal number of triple points among all quasi-filling Dehn surfaces of $M$, {\em i.e.}~$|T(\Sigma)|=sc(M)$.

\begin{teo}\label{teo:minimal_filling}
Let $M$ be a (connected and closed) \ptwoirred\ 3-manifold.
\begin{itemize}
\item
If $sc(M)=0$, then $M$ is the sphere $S^3$, the projective space $\RP^3$ or the lens space $\Lfourone$.
\item
If $sc(M)>0$, then $M$ has a minimal filling Dehn surface.
\end{itemize}
\end{teo}

\begin{proof}
Let $\Sigma$ be a minimal quasi-filling Dehn surface of $M$.
If we have $S(\Sigma)=\emptyset$ ({\em i.e.}~$\Sigma$ is a surface), by virtue of Remark~\ref{rem:Sigma_surf}, we have that $M$ is the sphere $S^3$ or the projective space $\RP^3$.

Then, we suppose $S(\Sigma)\neq\emptyset$.
We will first prove that $M$ has a quasi-filling Dehn surface $\Sigma'$ such that
$\Sigma' \setminus S(\Sigma')$ is made up of discs.
In fact, suppose there exists a component $C$ of $\Sigma \setminus S(\Sigma)$ that is not a disc.
$C$ contains a non-trivial orientation preserving (in $C$) simple closed curve $\gamma$.
Consider a strip $A$ contained in a small regular neighbourhood $\calR(\Sigma)$ of $\Sigma$ in $M$ such that $A\cap\Sigma=\gamma$ and $A\cap\partial\calR(\Sigma)=\partial A$.
(Note that $A$ is an annulus or a M\"obius strip depending on whether $\gamma$ is orientation preserving in $M$ or not.)
Since $M\setminus\calR(\Sigma)$ is made up of balls, we can fill up $\partial A$ with one or two discs disjoint from $\Sigma$ getting a sphere or a projective plane (depending on whether $A$ is an annulus or a M\"obius strip).
Note that both the sphere and the projective plane are transversely orientable, hence the second case cannot occur (because $M$ is \ptwoirred).
In the first case, since $M$ is \ptwoirred, the sphere found bounds a ball, say $B$.
Since $B$ is a ball and $\Sigma\cap\partial B=\gamma$ is a simple closed curve, we can replace the portion of $\Sigma$ contained in $B$ with a disc, getting a new quasi-filling Dehn surface of $M$.
Note that the Euler characteristic of the component of $\Sigma \setminus S(\Sigma)$ containing $\gamma$ has increased, that no new non-disc component has been created and that the number of triple points has not changed.
Hence, by repeatedly applying this procedure, we eventually get a quasi-filling Dehn surface, say $\Sigma'$, of $M$ such that $\Sigma' \setminus S(\Sigma')$ is made up of discs.

Since $\Sigma'$ is connected and $\Sigma' \setminus S(\Sigma')$ is made up of discs, we have that $S(\Sigma')$ is also connected.
If we have $sc(M)>0$ ({\em i.e.}~$T(\Sigma')$ is not empty), $S(\Sigma) \setminus T(\Sigma)$ cannot contain circles and hence $\Sigma'$ is filling ({\em i.e.}~$M$ has a minimal filling Dehn surface).
Otherwise, if we have $sc(M)=0$ ({\em i.e.}~$T(\Sigma')$ is empty), $S(\Sigma)$ is made up of one circle.
Since $\Sigma' \setminus S(\Sigma')$ is made up of discs, the Dehn surface $\Sigma'$ is completely determined by the regular neighbourhood of $S(\Sigma')$ in $\Sigma'$.
This neighbourhood depends on how the germs of disc are interchanged along the curve $S(\Sigma)$.
Among all possibilities we must rule out those not yielding a quasi-filling Dehn surface, hence only three ones must be taken into account (up to symmetry):
\begin{itemize}
\item two spheres intersecting along the circle $S(\Sigma)$ which form a Dehn surface of $S^3$; \item the double projective plane $\twoRPtwo$ which is a Dehn surface of $\RP^3$;
\item the four-hat which is a Dehn surface of $\Lfourone$.
\end{itemize}
The proof is complete.
\end{proof}

Since there is a finite number of filling Dehn surfaces having a fixed number of triple points, we have the following corollary of Theorem~\ref{teo:minimal_filling}.

\begin{cor}
For any integer $c$ there exists only a finite number of (connected and closed) \ptwoirred\ 3-manifolds having surface-complexity $c$.
\end{cor}

\subsection{Minimal quasi-filling Dehn surfaces}

Not all the minimal quasi-filling Dehn surfaces of \ptwoirred\ 3-manifolds are indeed filling.
However, they can be all constructed starting from filling ones (except for $S^3$, $\RP^3$ and $\Lfourone$, for which non-filling ones must be used) and applying a simple move.
The move acts on quasi-filling Dehn surfaces near a simple point as shown in Fig.~\ref{fig:bubble_move} and it is called {\em bubble-move}.
\begin{figure}[t]
  \centerline{\includegraphics{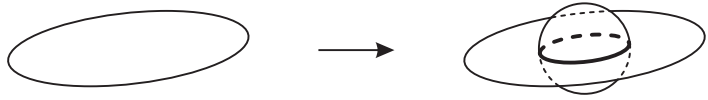}}
  \caption{Bubble-move.}
  \label{fig:bubble_move}
\end{figure}
Note that the result of applying a bubble-move to a quasi-filling Dehn surface of $M$ is a quasi-filling Dehn surface of $M$, but the result of applying a bubble-move to a filling Dehn-surface is not a filling Dehn-surface.
Note also that the bubble-move increases (by two) the number of connected components of $M\setminus\Sigma$.
If a quasi-filling Dehn surface $\Sigma$ is obtained from a quasi-filling Dehn surface $\overline{\Sigma}$ by repeatedly applying bubble-moves, we will say that $\Sigma$ {\em is derived from} $\overline{\Sigma}$.
Note that if $\Sigma$ is a quasi-filling Dehn surface of $M$ and is derived from $\overline{\Sigma}$, than $\overline{\Sigma}$ is a quasi-filling Dehn surface of $M$.

Theorem~\ref{teo:minimal_filling} can be improved by means of a slightly subtler analysis.

\begin{lem}\label{lem:all_minimal_filling_sphere}
Let $\Sigma$ be a minimal quasi-filling Dehn surface of the sphere $S^3$ and let $D$ be a closed disc disjoint from the singular set of $\Sigma$.
Then $\Sigma$ is derived from a sphere $S^2$ by means of bubble-moves not involving $D$.
\end{lem}

\begin{proof}
Since the surface complexity of $S^3$ is zero, the number of triple points of $\Sigma$ is zero and hence the connected components of $S(\Sigma)$, if there is any, are simple closed curves.
If we have $S(\Sigma)=\emptyset$ ({\em i.e.}~$\Sigma$ is a surface), by virtue of Remark~\ref{rem:Sigma_surf}, we have that $\Sigma$ is the sphere $S^2$.
Then, we will suppose $S(\Sigma)\neq\emptyset$ and we will prove the statement by induction on the number of connected components of $S(\Sigma)$.

Suppose that $S(\Sigma)$ has one connected component.
We will firstly prove that $\Sigma \setminus S(\Sigma)$ is made up of discs.
In fact, suppose by contradiction that there exists a component $C$ of $\Sigma \setminus S(\Sigma)$ that is not a disc.
$C$ contains a non-trivial orientation preserving (in $C$) simple closed curve $\gamma$.
Consider a strip $A$ contained in a small regular neighbourhood $\calR(\Sigma)$ of $\Sigma$ in $S^3$ such that $A\cap\Sigma=\gamma$ and $A\cap\partial\calR(\Sigma)=\partial A$.
Note that $A$ is an annulus because $\gamma$ is orientation preserving in $S^3$.
Since $S^3\setminus\calR(\Sigma)$ is made up of balls, we can fill up the annulus $\partial A$ with two discs disjoint from $\Sigma$ getting a sphere.
This sphere bounds two balls, say $B_1$ and $B_2$.
Since $S(\Sigma)$ does not intersect the disconnecting sphere, we have that $S(\Sigma)$ is wholly contained either in $B_1$ or in $B_2$ (we can assume in $B_1$).
Hence, $\Sigma \cap B_2$ is a surface cutting $B_2$ up into two balls, whose boundaries contain $\Sigma \cap B_2$.
Since $\Sigma \cap B_2$ has only one boundary component ($\gamma$), it is a disc and hence $\gamma$ is trivial in $C$, a contradiction.
We have proved that $\Sigma \setminus S(\Sigma)$ is made up of discs.
Since $S(\Sigma)$ is connected and does not contain triple points, it is a circle and the Dehn surface $\Sigma$ is completely determined by the regular neighbourhood of $S(\Sigma')$ in $\Sigma$.
This neighbourhood depends on how the germs of disc are interchanged along the curve $S(\Sigma)$.
Among all possibilities, only one yields a quasi-filling Dehn surface of $S^3$:
more precisely, $\Sigma$ is composed of two spheres intersecting along the circle $S(\Sigma)$.
We conclude by noting that $\Sigma$ is derived from the sphere $S^2$ by means of a bubble-move not involving $D$.

Finally, suppose that $S(\Sigma)$ has $n$ components with $n>1$ and suppose that the statement is true for all minimal quasi-filling Dehn surfaces of $S^3$ whose singular set has less than $n$ components.
Consequently, $S(\Sigma)$ is not connected and there is a connected component of $\Sigma \setminus S(\Sigma)$ that is not a disc.
This component contains a non-trivial orientation preserving (in $C$) simple closed curve $\gamma$ disjoint from the disc $D$.
As done above, we can construct a sphere $S^2$ intersecting $\Sigma$ along $\gamma$.
Cutting up $S^3$ by this sphere, we obtain two balls (say $B_1$ and $B_2$) both of which contain some components of $S(\Sigma)$.
Note that the disc $D$ is wholly contained either in $B_1$ or in $B_2$ (we can assume in $B_1$).
Consider now $\Sigma_2 = \Sigma \cap B_2$.
If we fill up $\Sigma_2$ with a disc (say $D_2$) by gluing it along $\gamma$, we obtain a minimal quasi-filling Dehn surface $\Sigma_2'$ of $S^3$ such that $S(\Sigma_2')$ has less than $n$ components.
We can apply the inductive hypothesis and we have that $\Sigma_2'$ is derived from a sphere $S^2$ by means of bubble-moves not involving $D_2$.
Since these moves do not involve the disc $D_2$, we can repeat these moves on $\Sigma$ obtaining a minimal quasi-filling Dehn surface $\Sigma_1'$ of $S^3$ such that $S(\Sigma_1')$ has less than $n$ components.
Note that all moves do not involve the disc $D$.
By applying again the inductive hypothesis, we obtain that $\Sigma_1'$ is derived from a sphere $S^2$ by means of bubble-moves not involving $D$.
Summing up $\Sigma$ is derived from a sphere $S^2$ by means of bubble-moves not involving $D$.
This concludes the proof.
\end{proof}

We are now able to prove the theorem that tells us how to construct all minimal quasi-filling Dehn surfaces starting from the filling ones (except for $S^3$, $\RP^3$ and $\Lfourone$, for which non-filling ones must be used).
\begin{teo}
Let $\Sigma$ be a minimal quasi-filling Dehn surface of a (connected and closed) \ptwoirred\ 3-manifold $M$.
\begin{itemize}
\item
If $sc(M)=0$, one of the following holds:
\begin{itemize}
\item
$M$ is the sphere $S^3$ and $\Sigma$ is derived from the sphere $S^2$,
\item
$M$ is the projective space $\RP^3$ and $\Sigma$ is derived from the projective plane $\RP^2$ or from the double projective plane $\twoRPtwo$,
\item
$M$ is the lens space $\Lfourone$ and $\Sigma$ is derived from the four-hat.
\end{itemize}
\item
If $sc(M)>0$, then $\Sigma$ is derived from a minimal filling Dehn surface of $M$.
\end{itemize}
\end{teo}

\begin{proof}
The scheme of the proof is the same as that of Theorem~\ref{teo:minimal_filling}.
Hence, we will often refer to the proof of Theorem~\ref{teo:minimal_filling} also for notation.

Let $\Sigma$ be a minimal quasi-filling Dehn surface of $M$.
If we have $S(\Sigma)=\emptyset$ ({\em i.e.}~$\Sigma$ is a surface), by virtue of Remark~\ref{rem:Sigma_surf}, we have that $\Sigma$ is the sphere $S^2$ or the projective plane $\RP^2$, and that $M$ is the sphere $S^3$ or the projective space $\RP^3$, respectively.

Then, we suppose $S(\Sigma)\neq\emptyset$.
We will first prove that $\Sigma$ is derived from a (minimal) quasi-filling Dehn surface $\Sigma'$ of $M$ such that either $\Sigma' \setminus S(\Sigma')$ is made up of discs or $\Sigma'$ is a surface.
In fact, suppose there exists a component $C$ of $\Sigma \setminus S(\Sigma)$ that is not a disc.
$C$ contains a non-trivial orientation preserving (in $C$) simple closed curve $\gamma$.
As done in the proof of Theorem~\ref{teo:minimal_filling}, we can construct a sphere $S^2$ contained in $M$ such that $S^2\cap\Sigma=\gamma$.
Since $M$ is \ptwoirred, the sphere found bounds a ball, say $B$.
Consider now $\Sigma_1 = \Sigma \setminus B$ and $\Sigma_2 = \Sigma \cap B$.
If we fill up $\Sigma_1$ with a disc by gluing it along $\gamma$, we obtain a minimal quasi-filling Dehn surface $\Sigma_1'$ of $M$.
Analogously, if we fill up $\Sigma_2$ with a disc (say $D$) by gluing it along $\gamma$, we obtain a minimal quasi-filling Dehn surface $\Sigma_2'$ of $S^3$.
By virtue of Lemma~\ref{lem:all_minimal_filling_sphere}, $\Sigma_2'$ is derived from a sphere $S^2$ by means of bubble-moves not involving $D$.
These moves can by applied to $\Sigma_1'$ because they do not involve $D$.
Note that the Euler characteristic of the component of $\Sigma \setminus S(\Sigma)$ containing $\gamma$ has increased, that no new non-disc component has been created and that the number of triple points has not changed.
Hence, by repeatedly applying this procedure, we eventually get a (minimal) quasi-filling Dehn surface $\Sigma'$ of $M$ from which $\Sigma$ is derived and such that either $\Sigma' \setminus S(\Sigma')$ is made up of discs or $\Sigma'$ is a surface.

If $\Sigma'$ is a surface, by virtue of Remark~\ref{rem:Sigma_surf}, $\Sigma'$ is either $S^2$ or $\RP^2$, and $M$ is $S^3$ or $\RP^3$, respectively; therefore, we have done.
Then, we suppose that $\Sigma' \setminus S(\Sigma')$ is made up of discs.
Since $\Sigma'$ is connected, we have that $S(\Sigma')$ is also connected.
If we have $sc(M)>0$ ({\em i.e.}~$T(\Sigma')$ is not empty), $S(\Sigma) \setminus T(\Sigma)$ cannot contain circles and hence $\Sigma'$ is filling ({\em i.e.}~$\Sigma$ is derived from a minimal filling Dehn surface of $M$).
Otherwise, if we have $sc(M)=0$ ({\em i.e.}~$T(\Sigma')$ is empty), $S(\Sigma)$ is made up of one circle.
Since $\Sigma' \setminus S(\Sigma')$ is made up of discs, the Dehn surface $\Sigma'$ is completely determined by the regular neighbourhood of $S(\Sigma')$ in $\Sigma'$.
This neighbourhood depends on how the germs of disc are interchanged along the curve $S(\Sigma')$.
Among all possibilities we must rule out those not yielding a quasi-filling Dehn surface, hence only three ones must be taken into account (up to symmetry):
\begin{itemize}
\item two spheres intersecting along the circle $S(\Sigma')$ which form a Dehn surface of $S^3$; \item the double projective plane $\twoRPtwo$ which is a Dehn surface of $\RP^3$;
\item the four-hat which is a Dehn surface of $\Lfourone$.
\end{itemize}
Note that in the first case $\Sigma'$ is derived from the sphere $S^2$.
Therefore, $\Sigma$ is derived from $S^2$, $\twoRPtwo$ or the four hat.
The proof is complete.
\end{proof}

\section{Cubulations}

A {\em cubulation} of $M$ is a cell-decomposition of $M$ such that
\begin{itemize}
\item each 2-cell (called {\em face}) is glued along 4 edges,
\item each 3-cell (called {\em cube}) is glued along 6 faces arranged like the
  boundary of a cube.
\end{itemize}
Note that self-adjacencies and multiple adjacencies are allowed.
In Fig.~\ref{fig:cubul_example} we have shown a cubulation of the 3-dimensional torus
$S^1\times S^1\times S^1$ with two cubes (the identification of each pair of
faces is the obvious one, {\em i.e.}~the one without twists).
\begin{figure}[t]
  \centerline{\includegraphics{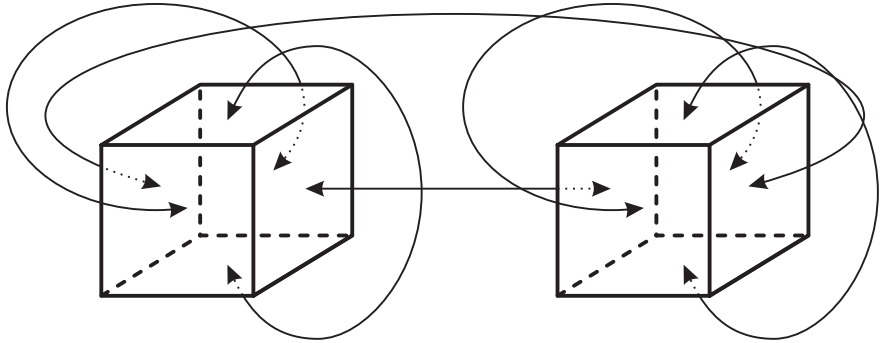}}
  \caption{A cubulation of the 3-dimensional torus $S^1\times S^1\times S^1$
    with two cubes (the identification of each pair of faces is the obvious
    one, {\em i.e.}~the one without twists).}
  \label{fig:cubul_example}
\end{figure}

The following construction is well-known (see~\cite{Aitchison-Matsumotoi-Rubinstein, Funar, Babson-Chan}, for instance).
Let $\calC$ be a cubulation of a closed 3-manifold; consider, for each cube of
$\calC$, the three squares shown in Fig.~\ref{fig:cube_to_surf};
the subset of $M$ obtained by gluing together
all these squares is a filling Dehn surface $\Sigma$ of $M$ (up to isotopy, we can suppose that the squares fit together through the faces).
\begin{figure}[t]
  \centerline{\includegraphics{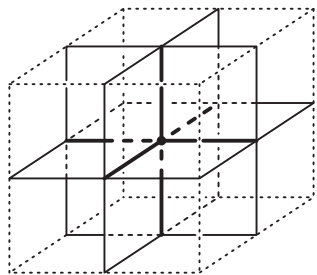}}
  \caption{Local behaviour of duality.}
  \label{fig:cube_to_surf}
\end{figure}
Conversely, a cell-decomposition $\calC$ can be constructed from a filling
Dehn surface $\Sigma$ of $M$ by considering an abstract cube for each triple
point of $\Sigma$ and by gluing the cubes together along the faces (the identification of each pair of faces is chosen by following the four germs of regions adjacent
to the respective edge of $\Sigma$); the cell-decomposition $\calC$ just constructed is indeed a cubulation of $M$.
The cubulation and the filling Dehn surface constructed in such a way are said
to be {\em dual} to each other.

An obvious corollary of Theorem~\ref{teo:minimal_filling} is the following result.

\begin{cor}
The surface-complexity of a (connected and closed) \ptwoirred\ 3-manifold, different from 
$S^3$, $\RP^3$ and $\Lfourone$, is equal to the minimal number of cubes in a cubulation of $M$.
\end{cor}

\section{Subadditivity}\label{sec:subadditivity}

An important feature of a complexity function is to behave well with respect to the
cut-and-paste operations.
In this section, we will prove that the surface-complexity is subadditive under connected sum.
We do not know whether it is indeed additive.

\begin{teo}\label{teo:sub_additivity}
The complexity of the connected sum of (connected and closed) 3-manifolds is less than or equal to the sum of their complexities.
\end{teo}

\begin{proof}
In order to prove the theorem, it is enough to prove the statement in the case where the number of the manifolds involved in the connected sum is two.
Hence, if we call $M_1$ and $M_2$ the two manifolds, we need to prove that
$sc(M_1\# M_2) \leqslant sc(M_1) + sc(M_2)$.
Let $\Sigma_1$ (resp.~$\Sigma_2$) be a quasi-filling Dehn surface of $M_1$ (resp.~$M_2$) with $sc(M_1)$ (resp.~$sc(M_2)$) triple points.
If the balls we remove to obtain the connected sum are disjoint from the $\Sigma_i$'s, we can suppose that $\Sigma_1$ and $\Sigma_2$ are embedded also in the connected sum $M_1\# M_2$.
All the components of $(M_1\# M_2) \setminus (\Sigma_1 \cup \Sigma_2)$ are balls except one that is a product $S^2 \times [0,1]$; see Fig.~\ref{fig:connected_sum}-left.
\begin{figure}[t]
  \centerline{\includegraphics{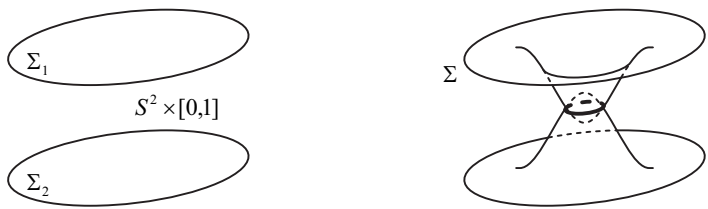}}
  \caption{The Dehn surface $\Sigma_1 \cup \Sigma_2$ in $M_1\# M_2$ (left) and its modification $\Sigma$ being quasi-filling (right).}
  \label{fig:connected_sum}
\end{figure}
We modify $\Sigma_1 \cup \Sigma_2$ as shown in Fig.~\ref{fig:connected_sum}-right, getting a Dehn surface, say $\Sigma$.
The complement $(M_1\# M_2) \setminus \Sigma$ is made up of the same balls as before (up to isotopy), a new small ball and a product $D^2 \times [0,1]$ (which is indeed a ball).
Therefore, $\Sigma$ is a quasi-filling Dehn surface of $M_1\# M_2$.
Since $\Sigma$ has $sc(M_1)+sc(M_2)$ triple points, we have $sc(M_1\# M_2) \leqslant sc(M_1) + sc(M_2)$.
\end{proof}

\section{Estimations}

\subsection{Matveev complexity}\label{subsec:Matveev_complexity}

The Matveev complexity~\cite{Matveev:compl} of a closed 3-manifold $M$ is defined using simple spines.
A polyhedron $P$ is {\em simple} if the link of each point of $P$ can be embedded in the 1-skeleton of the tetrahedron.
The points of $P$ whose link is the whole 1-skeleton of the tetrahedron are called {\em vertices}.
A sub-polyhedron $P$ of $M$ is a {\em spine} of $M$ if $M \setminus P$ is a ball.
The {\em Matveev complexity} $c(M)$ of $M$ is the minimal number of vertices of a simple spine of $M$.
The Matveev complexity is a natural measure of how much the manifold is complicated, because if $M$ is \ptwoirred\ and different from the sphere $S^3$, the projective space $\RP^3$ and the Lens space $\Lthreeone$, then its Matveev complexity is the minimum number of tetrahedra in a triangulation of it (the Matveev complexity of $S^3$, $\RP^3$ and $\Lthreeone$ is zero).
A simple spine of $M$ is {\em standard} if it is purely 2-dimensional and its singularities induce a cell-decomposition of $M$.
The dual cellularization of a standard spine of $M$ is a one-vertex triangulation of $M$, see~\cite{Matveev:book}.

The Matveev complexity is related to the surface-complexity.
Before describing more precisely this relation, we describe two constructions allowing us to create standard spines (or one-vertex triangulations, by duality) from filling Dehn surfaces (or cubulations, by duality), and {\em vice versa}.

Let $\calT$ be a one-vertex triangulation of $M$.
Consider, for each tetrahedron of $\calT$, the four triangles shown in Fig.~\ref{fig:tria_to_surf}.
\begin{figure}[t]
  \centerline{\includegraphics{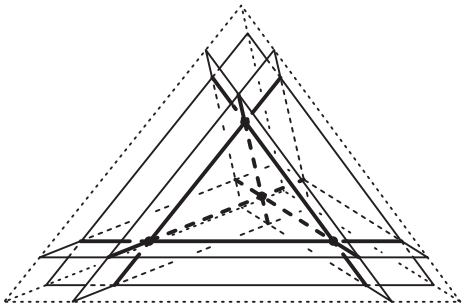}}
  \caption{Construction of a nullhomotopic filling Dehn sphere from a one-vertex triangulation.}
  \label{fig:tria_to_surf}
\end{figure}
The subset of $M$ obtained by gluing together all these triangles is a Dehn surface $\Sigma$ of $M$ with $4c(M)$ triple points (up to isotopy, we can suppose that the triangles fit together through the faces).
It is very easy to prove that $\Sigma$ is filling, so we leave it to the reader.
The construction just described is the dual counterpart of the well-known construction consisting in dividing a tetrahedron into four cubes~\cite{Shtanko-Shtogrin, Dolbilin-Shtanko-Shtogrin, Funar}.

Conversely, let $\calC$ be a cubulation of $M$.
Consider, for each cube of $\calT$, the five tetrahedra shown in Fig.~\ref{fig:cube_to_tetra}.
\begin{figure}[t]
  \centerline{\includegraphics{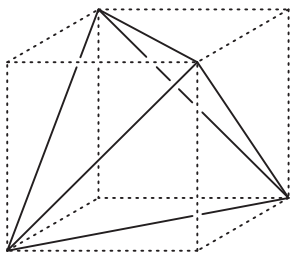}}
  \caption{Construction of a triangulation from a cubulation.}
  \label{fig:cube_to_tetra}
\end{figure}
The idea is to glue together these ``bricks'' (each of which is made up of five tetrahedra) by following the identifications of the faces of $\calC$.
Note that the faces of the cubes are divided by diagonals into two triangles and that it may occur that these pairs of triangles do not match each other.
If they do not match each other, we insert a tetrahedron between them as shown in Fig.~\ref{fig:insert_tetra}.
\begin{figure}[t]
  \centerline{\includegraphics{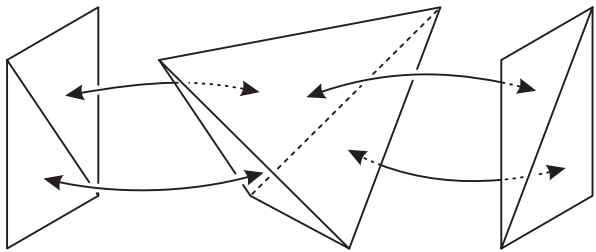}}
  \caption{Inserting a tetrahedron between two pairs of triangles not matching each other.}
  \label{fig:insert_tetra}
\end{figure}
Eventually, we get a triangulation $\calT$ of $M$ with $5$ tetrahedra for each cube of $\calC$ and at most one tetrahedron for each face of $\calC$.
Since the number of faces of a cubulation is thrice the number of cubes, we have that the triangulation $\calT$ we have constructed has at most $8sc(M)$ tetrahedra.

We note that there are two different identifications of the abstract ``brick'' with each cube, so there are $2^{sc(M)}$ possibilities for the identifications with the cubes of $\calC$.
Some of them may need less insertions of tetrahedra (for matching the pairs of triangles in the faces of $\calC$) than others.
Hence, optimal choices may lead to a triangulation of $M$ with $5sc(M)$ tetrahedra or few more.

The two constructions above and the list of the (connected and closed) \ptwoirred\ 3-manifolds with $c=0$ or $sc=0$ obviously imply the following.

\begin{teo}\label{teo:esimation_matveev-surface}
Let $M$ be a (connected and closed) \ptwoirred\ 3-manifold different from $\Lthreeone$ and $\Lfourone$; then we have
$$
sc(M) \leqslant 4c(M)
\quad \mbox{and} \quad
c(M) \leqslant 8sc(M).
$$
Moreover, we have $c(\Lthreeone)=0$, $sc(\Lthreeone)>0$, $c(\Lfourone)>0$ and $sc(\Lfourone)=0$.
\end{teo}

\subsection{Other estimations}

In general, calculating the surface-complexity $sc(M)$ of $M$ is very difficult, however it is relatively easy to estimate it.
More precisely, it is quite easy to give upper bounds for it.
If we construct a quasi-filling Dehn surface $\Sigma$ of $M$, the number of triple points of $\Sigma$ is an upper bound for the surface-complexity of $M$.
Afterwards, the (usually difficult) problem of proving the sharpness of this bound arises.

We can construct quasi-filling Dehn surfaces of $M$ from many presentations of $M$ and hence we obtain estimates from many presentations of $M$.
We use here three presentations: triangulations, Heegaard splittings and Dehn surgery.

We have already constructed a filling Dehn surface of $M$ from a one-vertex triangulation of $M$ in Section~\ref{subsec:Matveev_complexity}.
The same construction applies to any triangulation of $M$, yielding the following result.
\begin{teo}
Suppose a closed 3-manifold $M$ has a triangulation with $n$ tetrahedra. Then, we have $sc(M) \leqslant 4n$.
\end{teo}
\begin{proof}
Let $\calT$ be the triangulation of $M$ with $n$ tetrahedra.
Consider, for each tetrahedron of $\calT$, the four triangles shown in Fig.~\ref{fig:tria_to_surf}.
The subset of $M$ obtained by gluing together all these triangles is a Dehn surface $\Sigma$ of $M$ with $4n$ triple points.
It is very easy to prove that $\Sigma$ is filling, so we leave it to the reader.
\end{proof}

\begin{teo}\label{teo:heegaard_estimation}
Suppose $M = H_1 \cup H_2$ is a Heegaard splitting of a closed 3-manifold $M$ such that the meridians of the handlebody $H_1$ intersect those of $H_2$ transversely in $n$ points.
Then, we have $sc(M) \leqslant 4n$.
\end{teo}
\begin{proof}
Let $g$ be the number of meridians of $H_*$ and $H$ be the common boundary of $H_1$ and $H_2$.
Let $\mu_1^i$ (resp.~$\mu_2^i$), with $i=1,\ldots,g$, be the meridians of $H_1$ (resp.~$H_2$).

Let us suppose at first that each meridian of $H_1$ intersects at least one of $H_2$, and {\em vice versa}.
Let us call $D_j^i$ be the disc bounded by $\mu_j^i$ in $H_j$.
The boundary of a small regular neighbourhood of a disc $D_j^i$ is a sphere (say $S_j^i$) intersecting $H$ transversely along two loops parallel to $\mu_j^i$, see Fig.~\ref{fig:meridian}.
\begin{figure}[t]
  \centerline{\includegraphics{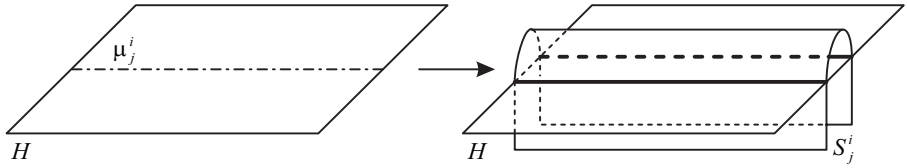}}
  \caption{The sphere $S_j^i$ near $H$.}
  \label{fig:meridian}
\end{figure}
The union of $H$ and the spheres $S_*^*$ is a Dehn surface, say $\Sigma$; we show it near an intersection point between a meridian of $H_1$ and one of $H_2$ in Fig.~\ref{fig:meridian_intersection}.
\begin{figure}[t]
  \centerline{\includegraphics{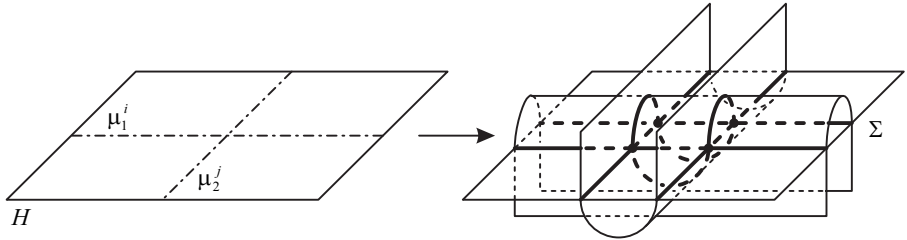}}
  \caption{The Dehn sphere $\Sigma$ near an intersection point between a meridian of $H_1$ and one of $H_2$.}
  \label{fig:meridian_intersection}
\end{figure}
We prove now that $\Sigma$ is quasi-filling.
Since $H$ is contained in $\Sigma$ and since the role of the two handlebodies is symmetrical, it is enough to prove that $H_1\setminus\Sigma$ is made up of balls.
The spheres $S_1^*$ divide $H_1$ into balls, because the discs $D_1^*$ do and each $S_1^i \cap H_1$ is made up of two discs parallel to $D_1^i$.
Moreover, the spheres $S_2^*$ divide these balls into smaller balls, because each meridian of $H_2$ intersects at least one of $H_1$.
Each point of intersection of two meridians yields 4 triple point (see Fig.~\ref{fig:meridian_intersection}), hence the number of triple points of $\Sigma$ is $4n$.
Therefore, we have proved that if each meridian of $H_1$ intersects at least one of $H_2$, and {\em vice versa}, then we have $sc(M) \leqslant 4n$.

Suppose now that some meridian of $H_1$ does not intersect any of $H_2$ (the case of a meridian of $H_2$ is symmetrical).
We can suppose, without loss of generality, that this meridian is $\mu_1^1$.
This Heegaard splitting is reducible.
A loop in $H$ parallel to $\mu_1^1$ bounds in fact a disc both in $H_1$ and in $H_2$.
Let us call these discs $D_1$ and $D_2$, respectively.
We can suppose that $D_1$ is parallel to $D_1^1$.
The disc $D_1$ does not disconnect $H_1$, therefore the sphere $D_1 \cup D_2$ does not disconnect $M$.
Hence, $M$ is the connected sum of $S^2 \times S^1$ (or $S^2 \timtil S^1$) and another manifold, say $M'$.
Moreover, we can explicitly construct a Heegaard splitting of $M'$.
Namely, the two handlebodies, say $H'_j$, are obtained by cutting $H_j$ along $D_j$ (for $j=1,2$); the gluing map coincides with the old one out of the two pairs of discs created in the boundary of $H_i$ by the cut and identifies the four discs in pairs.
Moreover, we consider the class of meridians of $H'_1$ made up of $\mu_1^i$, with $i=2,\ldots,g$.
In order to get a class of meridians of $H'_2$, we consider the meridians $\mu_2^i$ discarding one of them: in order to choose the one to discard, we look at the two spheres with holes obtained by cutting the boundary of $H'_2$ along the meridians $\mu_2^i$ and we discard a $\mu_2^i$ that is adjacent to both spheres.
Note that a good choice of the $\mu_2^i$ to discard may yield a decrease of the number of intersections between the meridians; however, we only know that the number of intersections between the meridians does not increase.

If now some meridian of $H'_1$ does not intersect any of $H'_2$, or {\em vice versa}, then we repeat the procedure.
Eventually, we have that $M$ is the connected sum of some copies of $S^2 \times S^1$ (and $S^2 \timtil S^1$) and another manifold $M^{(l)}$.
Moreover, we obtain that $M^{(l)}$ has a Heegaard splitting such that each meridian of a handlebody intersects at least one of the other handlebody.
If there is no meridian, $M^{(l)}$ is the sphere $S^3$ and hence $sc(M^{(l)})=0 \leqslant 4n$; otherwise, we have constructed a quasi-filling Dehn surface of $M^{(l)}$ with 4 triple points for each intersection of the meridians of the Heegaard splitting of $M^{(l)}$.
The number of intersections cannot increase along the procedure, therefore we have $sc(M^{(l)}) \leqslant 4n$.
Since we have $sc(S^2 \times S^1)=0$ (and $sc(S^2 \timtil S^1)=0$), by virtue of Theorem~\ref{teo:sub_additivity}, we have $sc(M) \leqslant sc(M^{(l)}) \leqslant 4n$.
The proof is complete.
\end{proof}

\begin{rem}
In~\cite{Matveev:book} the following is proven:\\
{\em Suppose $M = H_1 \cup H_2$ is a Heegaard splitting of a closed 3-manifold $M$ such that the meridians of the handlebody $H_1$ intersect those of $H_2$ transversely in $n$ points.
Suppose also that the closure of one of the components into which the meridians of $H_1$ and $H_2$ divide $\partial H_1 = \partial H_2$ contains $m$ such points.
Then, we have $c(M) \leqslant n-m$.}

We can use this result to improve the estimation of $sc(M)$ of Theorem~\ref{teo:heegaard_estimation} if the decomposition of $M$ into prime manifolds contains only \ptwoirred\ closed 3-manifolds different from $\Lthreeone$.
Let $M$ be the connected sum of \ptwoirred\ manifolds $M_k$, such that no $M_k$ is $\Lthreeone$.
By virtue of the statement above, we have $c(M) \leqslant n-m$.
Since the Matveev complexity is additive under connected sum, we have
$\sum_k c(M_k) = c(M) \leqslant n-m$.
Note that Theorem~\ref{teo:esimation_matveev-surface} implies $sc(M_k) \leqslant 4c(M_k)$ for each $M_k$.
Then, by virtue of Theorem~\ref{teo:sub_additivity}, we have
$sc(M) \leqslant \sum_k sc(M_k) \leqslant \sum_k 4c(M_k) \leqslant 4n-4m$.
\end{rem}

\begin{teo}\label{teo:chir_to_surf}
Suppose $M$ is obtained by Dehn surgery along a framed link $L$ in $S^3$ (hence $M$ is orientable).
Moreover, suppose $L$ has a projection such that the framing is the blackboard one, there are $n$ crossing points and there are $m$ components containing no overpass.
Then, we have $sc(M) \leqslant 8n + 4m$.
\end{teo}
\begin{proof}
The projection plane can be regarded as a subset of a sphere $S^2$ contained in $S^3$.
We add a cylinder for each arc of the projection, as shown in Fig.~\ref{fig:cylinder}.
\begin{figure}[t]
  \centerline{\includegraphics{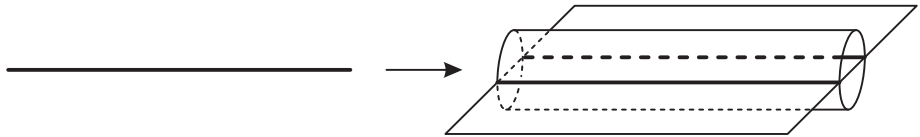}}
  \caption{Adding a cylinder for each arc of the projection.}
  \label{fig:cylinder}
\end{figure}
We connect these cylinders by a pair of intersecting cylinders for each crossing point, as shown in Fig.~\ref{fig:crossing}.
\begin{figure}[t]
  \centerline{\includegraphics{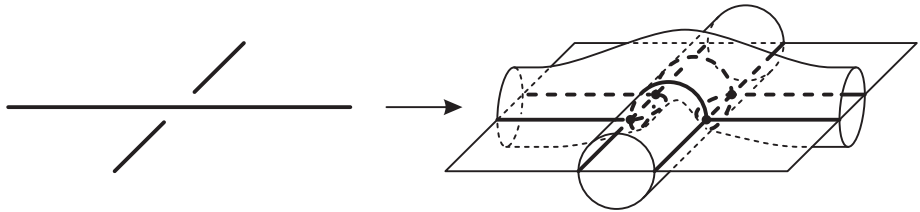}}
  \caption{Connecting the cylinders through the crossings.}
  \label{fig:crossing}
\end{figure}
The result is a Dehn surface, say $\Sigma$, contained in $S^3$ and made up of the sphere $S^2$ and some tori $T_i$ (namely, one torus for each component of $L$).

The complement of $\Sigma$ in $S^3$ is made up of balls and solid tori.
More precisely, we have one torus for each component of $L$ and one more torus for each component containing no overpass: note indeed that if $T_i$ is the torus corresponding to one of these components, then $S^3 \setminus (S^2 \cup T_i)$ is made up of two balls and two tori, both of which are not divided into balls by overpasses.
In order to divide them, we add one small sphere for each component containing no overpass, as shown in Fig.~\ref{fig:small_sphere}, getting another Dehn surface, say $\Sigma'$.
\begin{figure}[t]
  \centerline{\includegraphics{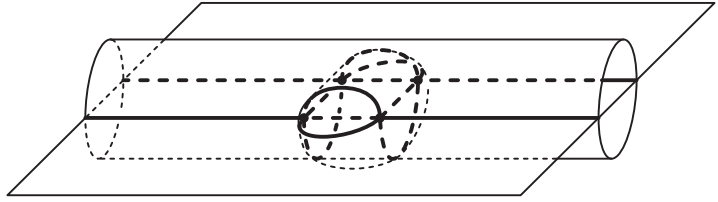}}
  \caption{Adding small spheres to divide the useless tori into balls.}
  \label{fig:small_sphere}
\end{figure}

Now, the complement of $\Sigma'$ in $S^3$ is made up of some balls and one torus for each component of the link $L$.
Moreover, up to isotopy, we can suppose that the union of the tori is a regular neighbourhood of the link $L$.
Hence, the Dehn surface $\Sigma'$ can be regarded as a Dehn surface in $M$.
Furthermore, the complement of $\Sigma'$ in $M$ is also made up of some balls and one torus, say $T'_i$, for each component of the link $L$.

In order to get a quasi-filling Dehn surface of $M$, we will divide the tori into balls by adding spheres.
For each torus $T_i$, let $\gamma_i$ be a curve giving the (blackboard) framing, disjoint from $S^2$ and lying above $S^2$ (with respect to the projection); see Fig.~\ref{fig:framing}.
\begin{figure}[t]
  \centerline{\includegraphics{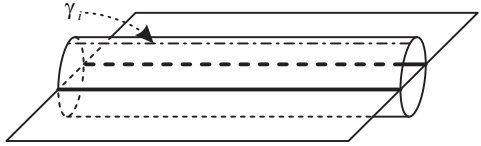}}
  \caption{The choice of the curves giving the (blackboard) framing.}
  \label{fig:framing}
\end{figure}
Moreover, consider a disc $D_i$ bounded in $T'_i$ by the curve $\gamma_i$ and consider the boundary of a small regular neighbourhood of the disc, say $S_i$.
Each surface $S_i$ is a sphere intersecting $\Sigma'$ along two curves (see Fig.~\ref{fig:framing_sphere}).
\begin{figure}[t]
  \centerline{\includegraphics{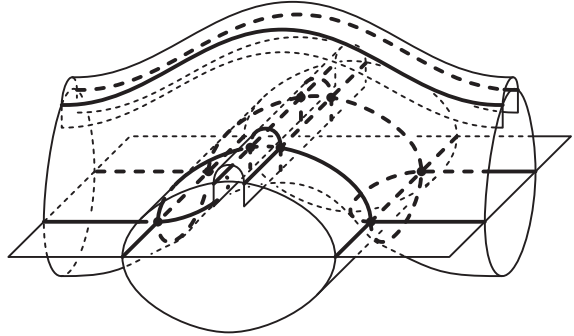}}
  \caption{The spheres dividing the tori into balls yield 4 triple points for each crossing point of the projection of $L$.}
  \label{fig:framing_sphere}
\end{figure}
We add the spheres $S_i$ to $\Sigma'$ and we call the result $\Sigma''$.
The complement of the Dehn surface $\Sigma''$ in $M$ is made up of balls, so $\Sigma''$ is a quasi-filling Dehn surface of $M$.
Moreover, $\Sigma''$ has 8 triple points for each crossing of the projection of $L$ (see Fig.~\ref{fig:framing_sphere}) and 4 triple points for each component of $L$ containing no overpass (see Fig.~\ref{fig:small_sphere}).
The proof is complete.
\end{proof}

\begin{rem}
If the framing is not the blackboard one, some curls must be added to the curves $\gamma_i$, as shown in Fig.~\ref{fig:curl}.
\begin{figure}[t]
  \centerline{\includegraphics{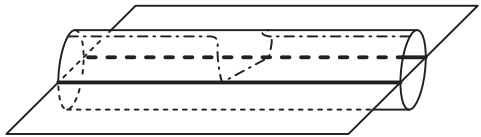}}
  \caption{A curl added if the framing is not the blackboard one.}
  \label{fig:curl}
\end{figure}
However, a simple generalisation of the procedure shown in the proof of Theorem~\ref{teo:chir_to_surf} yields an upper bound for surface-complexity by means of non-blackboard framings, where the wirthe ({\em i.e.}~the blackboard framing) appears.
More precisely, with the notation of Theorem~\ref{teo:chir_to_surf}, if $fr_i$ and $w_i$ are, respectively, the framing and the wirthe of the $i$-th component of the link, then we have
$sc(M) \leqslant 8n + 4m + 4\sum_i |fr_i-w_i|$.
\end{rem}

\appendix

\section{Starting enumeration}

We recall that there are eight important 3-dimensional geometries: six of them 
concern Seifert manifolds ($\matS^2\times\matR$, $\matE^3$, $\matH^2\times\matR$, $\matS^3$, Nil and $\widetilde{{\rm SL}_2\matR}$), one concerns Stallings manifolds (Sol) and the last is the hyperbolic one~\cite{Sco}.
The geometry of a Seifert manifold is determined by two invariants of any of its fibrations, namely the Euler characteristic $\chiorb$ of the base orbifold and the Euler number $e$ of the fibration, according to Table~\ref{tab:Seifert}.
\begin{table}[t]
	\centerline{\begin{tabular}{c|ccc} 
	\phantom{\Big|} & $\chiorb > 0$ & $\chiorb = 0$ & $\chiorb < 0$ \\ \hline
	\phantom{\Big|} $e = 0$ & $\matS^2\times\matR$ & $\matE^3$ & $\matH^2\times\matR$ \\
	\phantom{\Big|} $e \neq 0$ & $\matS^3$ & Nil & $\widetilde{{\rm SL}_2\matR}$ \\ 
	\end{tabular}}
	\caption{The six Seifert geometries.}
	\label{tab:Seifert}
\end{table}
Finally, we recall that non-orientable Seifert manifolds cannot have $e \neq 0$, {\em i.e.}~cannot be of type $\matS^3$, Nil and $\widetilde{{\rm SL}_2\matR}$.

In a subsequent paper~\cite{Amendola:next}, we will prove the following.
\begin{teo}
\begin{itemize}
\item
There are no \ptwoirred\ orientable closed 3-manifolds having surface-complexity one of type $\matH^2\times\matR$, $\widetilde{{\rm SL}_2\matR}$, Sol, hyperbolic or non-geometric.
\item
There are \ptwoirred\ orientable closed 3-manifolds having surface-com\-plexity one of type $\matS^3$ ({\em e.g.}~the lens spaces $L_{6,1}$, $L_{8,3}$, $L_{14,3}$, $L_{12,5}$) and $\matE^3$ ({\em e.g.}~the 3-dimensional torus $S^1 \times S^1 \times S^1$).
\item
There are no \ptwoirred\ non-orientable closed 3-manifolds having sur\-face-complexity one  of type $\matH^2\times\matR$, hyperbolic or non-geometric.
\item
There are \ptwoirred\ non-orientable closed 3-manifolds having surface-complexity one of type $\matE^3$ ({\em e.g.}~the trivial bundle over the Klein bottle $K \times S^1$).
\end{itemize}
\end{teo}

\section{The bidimensional case}

Let $S$ be a connected closed surface.
Let us denote by $(S^1)^{\sqcup n}$ the disjoint union of $n$ circles $S^1$.
A subset $\Gamma$ of $S$ is said to be a {\em Dehn loop of $S$}
if there exists $n\in\matN$ and a transverse immersion $f\co (S^1)^{\sqcup n}\to S$ such that $\Gamma = f((S^1)^{\sqcup n})$; some examples are shown in Fig.~\ref{fig:bidim_example}.
\begin{figure}[t]
  \centerline{\includegraphics{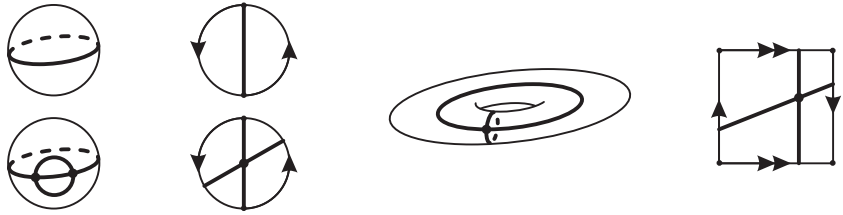}}
  \caption{Some Dehn loops of the sphere $S^2$, the projective plane $\RP^2$, the torus $T$ and the Klein bottle $K$.}
  \label{fig:bidim_example}
\end{figure}
There are only two types of points in $\Gamma$: smooth points and crossing points (see Fig.~\ref{fig:bidim_neigh}). The set of crossing points is denoted by $C(\Gamma)$.
\begin{figure}[t]
  \centerline{\includegraphics{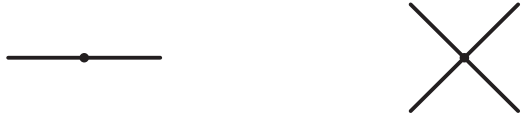}}
  \caption{Smooth points (left) and crossing points (right).}
  \label{fig:bidim_neigh}
\end{figure}
A Dehn loop $\Gamma$ of $S$ will be called {\em quasi-filling} if $S \setminus \Gamma$ is made up of discs; it will be called {\em filling} if its singularities induce a cell-decomposition of $S$: more precisely,
\begin{itemize}
\item $C(\Gamma) \neq \emptyset$,
\item $\Gamma \setminus C(\Gamma)$ is made up of intervals,
\item $S \setminus \Gamma$ is made up of discs ({\em i.e.}~$\Gamma$ is quasi-filling).
\end{itemize}
Note that the second condition is automatically satisfied once the other two are fulfilled.
It can be easily shown that only the sphere $S^2$ and the projective plane $\RP^2$ have quasi-filling Dehn loops that are not filling, while all non-positive Euler characteristic surfaces have only filling Dehn loops; see Fig.~\ref{fig:bidim_example}.
Each surface $S$ has a filling Dehn loop.
However, as opposed to the 3-dimensional case, a filling Dehn loop does not determine $S$; in fact, the bouquet of two circles is a Dehn loop of the projective plane $\RP^2$, the torus $T$ and the Klein bottle $K$; see Fig.~\ref{fig:bidim_example}.

We define the {\em loop-complexity} $lc(S)$ of a connected closed surface $S$ as the minimal number of crossing points of a quasi-filling Dehn loop of $S$.
The loop-complexity of the surface with Euler characteristic $\chi$ is $1-\chi$, except for the sphere $S^2$ having loop complexity zero.
In terms of the genus of the surface, we have:
\begin{itemize}
\item
the loop-complexity of the sphere $S^2$ and of the projective plane $\RP^2$ is zero,
\item
the loop-complexity of the orientable genus-$g$ surface is $2g-1$,
\item
the loop-complexity of the non-orientable genus-$g$ surface is $g-1$.
\end{itemize}
Note that it is not true that the loop-complexity of the connected sum of two surfaces is at most the sum of their loop-complexities; for instance, we have $K=\RP^2\#\RP^2$ while $lc(K)=1 \not\leqslant 0=lc(\RP^2)+lc(\RP^2)$.

A cubulation $\calC$ of a connected closed surface $S$ ({\em i.e.}~a cell-decomposition of $S$ such that each 2-cell, called {\em square}, is glued along 4 edges) can be constructed from a filling Dehn loop $\Gamma$ of $S$ by considering an abstract square for each crossing point of $\Gamma$ and by gluing the squares together along the edges.
However, there are two possibilities for gluing two squares along an edge and the abstract polyhedron $\Gamma$ does not encode any information to choose the right one.
(In some sense, this explains why a filling Dehn loop does not determine unambiguously one surface.)
If we consider also the immersion of $(S^1)^{\sqcup n}$ in the surface $S$ containing the Dehn loop $\Gamma$, we can choose the right identifications and construct a cubulation $\calC$ of $S$.

The converse construction can also be performed obtaining a filling Dehn loop of a connected closed surface $S$ from a cubulation of $S$.
These constructions tells us that $lc(S)$ is the minimal number of squares in a cubulation of $S$, except for $S^2$ and $\RP^2$ whose loop-complexity is zero.

\paragraph{Acknowledgements}
I would like to thank the Department of Mathematics of the University of Salento for the nice welcome and, in particular, Prof.~Giuseppe De Cecco for his willingness.

\begin{small}

\end{small}

\vspace{2cm}
\noindent
Department of Mathematics\\
University of Salento\\
Palazzo Fiorini, Via per Arnesano\\
I-73100 Lecce, Italy\\ 
amendola@mail.dm.unipi.it

\end{document}